\documentclass[11pt]{amsart}
\usepackage{graphicx,color,mathrsfs,amsmath,amssymb}

\theoremstyle{plain}

\theoremstyle{plain}
\newtheorem{theorem}{Theorem} [section]

\newtheorem{lemma}[theorem]{Lemma}
\newtheorem{proposition}[theorem]{Proposition}

\theoremstyle{definition}
\newtheorem{definition}[theorem]{Definition}

\theoremstyle{remark}

\newtheorem{remark}[theorem]{Remark}
\numberwithin{theorem}{section}
\numberwithin{equation}{section}
\numberwithin{figure}{section}

\def\mean#1{\mathchoice
         {\mathop{\kern 0.2em\vrule width 0.6em height 0.69678ex depth -0.58065ex
                 \kern -0.8em \intop}\nolimits_{\kern -0.4em#1}}%
         {\mathop{\kern 0.1em\vrule width 0.5em height 0.69678ex depth -0.60387ex
                 \kern -0.6em \intop}\nolimits_{#1}}%
         {\mathop{\kern 0.1em\vrule width 0.5em height 0.69678ex
             depth -0.60387ex
                 \kern -0.6em \intop}\nolimits_{#1}}%
         {\mathop{\kern 0.1em\vrule width 0.5em height 0.69678ex depth -0.60387ex
                 \kern -0.6em \intop}\nolimits_{#1}}}

\def\N{\mathbb N}

\def\R{\mathbb R}
\def\L{\mathscr L}

\def\e{\varepsilon}

\def\l{\lambda}

\newcommand{\re}{\mathbb{R}}
\newcommand{\n}{\mathbb{N}}

\newcommand{\diam}{d_\Omega}
\newcommand{\supp}{{\rm supp}}

\newcommand{\negint}{{\int\negthickspace\negthickspace\negthickspace-}}

\begin{document}

\title[Global existence for the semigeostrophic equations]{A global existence result\\
for the semigeostrophic equations\\ in three dimensional convex domains
}

\author[L.\ Ambrosio]{Luigi Ambrosio}
\address{Scuola Normale Superiore,
p.za dei Cavalieri 7, I-56126 Pisa, Italy}
\email{l.ambrosio@sns.it}

\author[M.\ Colombo]{Maria Colombo}
\address{Scuola Normale Superiore,
p.za dei Cavalieri 7, I-56126 Pisa, Italy}
\email{maria.colombo@sns.it}

\author[G.\ De Philippis]{Guido De Philippis}
\address{Scuola Normale Superiore,
p.za dei Cavalieri 7, I-56126 Pisa, Italy}
\email{guido.dephilippis@sns.it}

\author[A.\ Figalli]{Alessio Figalli}
\address{Department of Mathematics,
The University of Texas at Austin, 1 University Station C1200,
Austin TX 78712, USA}
\email{figalli@math.utexas.edu}

%

\begin{abstract}
Exploiting
recent regularity estimates for the Monge-Amp\`ere equation, under some suitable assumptions on
the initial data
we prove global-in-time existence of Eulerian distributional solutions to the semigeostrophic equations
in 3-dimensional convex domains . 
\end{abstract}

\maketitle

\section{Introduction}
A simplified model for the motion of large scale atmospheric/oceanic flows inside a domain $\Omega\subset \R^3$
is given by the semigeostrophic equations.

Let $\Omega\subseteq \R^3$ be bounded open set with Lipschitz boundary. Then the 
semigeostrophic equations inside $\Omega$ are:
\begin{equation}\label{eqn:SGsystem1}
\begin{cases}
\partial_t u^g_t(x) +\bigl(u_t(x) \cdot \nabla\bigr) u^g_t(x) + \nabla p_t(x) = -J u_t(x) +m_t(x) e_3
\quad\quad &(x,t)\in \Omega \times (0,\infty)\\
\partial_t m_t(x) +\bigl(u_t(x) \cdot \nabla\bigr)m_t(x) =0
\quad\quad &(x,t)\in \Omega \times (0,\infty)\\
u^g_t(x) = J \nabla p_t(x) & (x,t)\in \Omega \times [0,\infty)\\
\nabla \cdot u_t(x) = 0  &(x,t)\in \Omega \times [0,\infty)\\
u_t (x) \cdot \nu_\Omega(x) = 0 &(x,t)\in \partial \Omega \times [0,\infty)\\
p_0(x)= p^0(x)   &x\in \Omega.
\end{cases}
\end{equation}
Here $p^0$ is the initial condition for $p$,\footnote{As it will be clear from the discussion later,
we do not need to specify any initial condition for $u$ and $m$.} $\nu_\Omega$ is the unit outward normal to $\partial \Omega$, $e_3=(0,0,1)^T$ is the third vector of the canonical basis in $\R^3$, $J$ is the matrix given by
\[
J:=
\begin{pmatrix}
0  &  -1  &  0   \\
1  &   0  &  0  \\
0  &   0  &  0  \\
\end{pmatrix},
\]
and the functions $u_t$, $p_t$, and $m_t$ represent respectively the
\emph{velocity}, the \emph{pressure} and the \emph{density}  of the atmosphere, while $u^g_t$ is the
so-called \emph{semi-geostrophic} wind.\footnote{We are
using the notation $f_t$ to denote the function
$f(t,\cdot)$.} Clearly the pressure is
defined up to a (time-dependent) additive constant.

Substituting the relation $u^g_t = J \nabla p_t$ and introducing the function 
\begin{equation}\label{defn:P}
P_t(x):=p_t(x)+\frac{1}{2}(x_1^2+x_2^2),
\end{equation}
the system \eqref{eqn:SGsystem1} can be rewritten in $\Omega \times [0,\infty)$ as
\begin{equation}\label{eqn:SGsystem2}
\begin{cases}
\partial_t \nabla P_t (x) + \nabla^2 P_t(x) u_t(x) = J (\nabla P_t(x) -x)\\
\nabla \cdot u_t (x) = 0 \\
u_t (x) \cdot \nu_\Omega (x) = 0\\
P_0(x) = p^0(x) + \frac{1}{2}(x_1^2+x_2^2).
\end{cases}
\end{equation}

Notice that, given a solution $(P,u)$ of \eqref{eqn:SGsystem2}, one easily recovers a solution of
\eqref{eqn:SGsystem1}: indeed $p_t$ can be obtained from $P_t$ through \eqref{defn:P} and
the density $m_t$ is given by $m_t = \partial_3 P_t$ (in particular, the third
component of the first equation in \eqref{eqn:SGsystem2} tells us
that $\partial_t m_t +\bigl(u_t \cdot \nabla\bigr)m_t =0$ is satisfied).

Energetic considerations (see \cite[Section 3.2]{cu}) show that it
is natural to assume that the function $P_t$  is convex on $\Omega$. This condition, first introduced by Cullen and Purser, is related in \cite{CP, CS}
 to a physical stability required for the semigeostrophic approximation to be appropriate.
 If we denote with
$\L_{\Omega}$ the (normalized) Lebesgue measure on $\Omega$, then
formally $\rho_t:=(\nabla P_t)_\sharp \L_{\Omega}$ \ (see, for example,  \cite[Appendix A]{acdf})
satisfies the following \emph{dual problem}
\begin{equation}\label{eqn:dualsystem}
\begin{cases}
\partial_t \rho_t +\nabla \cdot (U_t \rho_t) = 0 \\
U_t(x) = J(x-\nabla P_t^*(x))\\
\rho_t= (\nabla P_t)_{\sharp} \L_{\Omega}\\
P_0(x)= p^0(x)+ \frac{1}{2}(x_1^2+x_2^2).
\end{cases}
\end{equation}
Here $P^*_t$ is the convex conjugate of $P_t$, namely
\[
P_t^*(y):=\sup_{x\in\Omega} (y \cdot x-P_t(x)) \qquad\quad \forall y\in \R^3
\]
and $(\nabla P_t)_{\sharp} \L_{\Omega}$ is the push-forward of the measure $\L_\Omega$ through the map $\nabla P_t :\Omega \to \R^3$ defined as
\[
[(\nabla P_t)_{\sharp} \L_{\Omega}](A)=\L_{\Omega}\big((\nabla P_t)^{-1}(A)\big)\qquad \text{for all $A\subset\R^3$ Borel.}
\]

The dual problem is pretty well understood, and admits a solution
obtained via time discretization (see \cite{bebr,CG}). Moreover,
at least formally, given a solution $P_t$ of the dual problem \eqref{eqn:dualsystem} and setting
\begin{equation}\label{defn:ut}
u_t(x):= [\partial_t\nabla P_t^*](\nabla P_t(x)) + [\nabla^2 P_t^*]( \nabla P_t(x)) J(\nabla P_t(x) - x),
\end{equation}
the couple $(P_t, u_t)$ solves the semi-geostrophic problem \eqref{eqn:SGsystem1}. However,
because of the low regularity of the function $P_t$
the previous velocity field may a priori not be well defined,
and this creates serious difficulties for recovering a ``real solution'' from a ``dual solution''.

Still, a recent regularity result \cite{DepFi} can be applied to show that the map $P_t$ is $W^{2,1}$
in space, so that we can give a meaning to the second term in the definition of $u_t$.
More precisely, in \cite{DepFi} it is shown that $|D^2u|\log^k_+|D^2u|\in L^1_{\rm loc}$ for any $k$,
and following ideas developed
in \cite{acdf,Loe1}, we will be able to show that the function $P_t$ is regular enough
also in time, so that the couple $(P_t, u_t)$ is a true distributional solution of \eqref{eqn:SGsystem1}.

Let us point out that the regularity result in \cite{DepFi}
has been recently  extended, independently  in \cite{DFS} and \cite{S}, to 
$|D^2 u|\in L^\gamma$, where $\gamma >1$ depends on the local $L^\infty$ norm of $\log \rho_t$.
However, as we will better explain in Remark \ref{exp2},
in our situation there is no advantage in using this improvement, since the fact that $\gamma$ depends on
$\|\log \rho_t\|_{\infty,\rm loc}$ makes the estimates less readable.
For this reason we will rely only on the $L\log L$ integrability given by
\cite{DepFi}, as we previously did in \cite{acdf}.

The first existence result about distributional solutions to
the semigeostrophic equation is presented in \cite{acdf}, where the analysis
is carried out on the 2-dimensional torus (see also \cite{Loe2} where a short time existence result of smooth solutions
is proved in dual variables, and because of smoothness the
existence can be easily transferred to the initial variables).

The 3-dimensional case on the whole space $\R^3$, which is more physically relevant, presents additional difficulties. First, the
equation \eqref{eqn:SGsystem1} is much less symmetric compared to its 2-dimensional counterpart, because the action of Coriolis force $Ju_t$ regards only the first and the second space components.
Moreover, even considering regular initial data and velocities, regularity results require a finer regularization scheme, due to the non-compactness of the ambient space.

Our proofs are also based on some additional hypotheses on the decay of the probability measure $\rho_0 =  (\nabla P_0)_{\sharp} \L_{\Omega}$. This decay condition happens to be stable in time on solutions of the dual equation \eqref{eqn:dualsystem},
and allows us to perform a regularization scheme.

It would be extremely interesting to consider compactly supported initial data
$\rho_0 =  (\nabla P_0)_{\sharp} \L_{\Omega}$. However the
nontrivial evolution of the support of the solution
$\rho_t$ under \eqref{eqn:dualsystem} prevents us to 
apply the results in \cite{DepFi}
(which actually would be false in this situation), so at the moment this case
seems to require completely new ideas and ingredients.

\begin{definition}\label{sg-weak-eul}
Let $P:\Omega \times [0,\infty)\to\R$ and $u:\Omega \times
[0,\infty)\to\R^3$. We say that $(P,u)$ is a \emph{weak Eulerian
solution} of \eqref{eqn:SGsystem2} if:
\begin{enumerate}
\item[-] $|u|\in L^\infty_{\rm loc}((0,\infty),L^1_{\rm loc}(\Omega))$,
$P\in L^\infty_{\rm loc}((0,\infty),W^{1,\infty}_{\rm loc}(\Omega))$, and
$P_t(x)$ is convex for any $t \geq 0$;
\item[-] For every
$\phi\in C^\infty_c(\Omega \times [0,\infty))$, it holds
\begin{multline}\label{eqn:sg1-weak}
 \int_0^\infty\int_{\Omega}
 \nabla P_t(x) \Big\{\partial_t \phi_t(x) + u_t(x)\cdot \nabla
 \phi_t(x)\Big\}+J \Big\{\nabla P_t(x)-x \Big\} \phi_t(x)\, dx \, dt
  +\int_{\Omega} \nabla P_0(x) \phi_0(x) \, dx = 0;
\end{multline}
\item[-] For a.e. $t\in (0,\infty)$ it holds
\begin{equation}\label{eqn:sg2-weak}
\int_{\Omega} \nabla \psi(x)\cdot u_t(x) \, dx =0 \qquad \mbox{for
all $\psi\in C^\infty_c({\Omega})$.}
\end{equation}
\end{enumerate}
\end{definition}

\begin{remark}This definition is the classical notion of distributional solution for
\eqref{eqn:SGsystem2} except
for the fact that the boundary condition $u_t \cdot \nu_\Omega = 0$ is not
taken into account.
In this sense it may look natural to consider $\psi\in
C^\infty(\overline{\Omega})$ in \eqref{eqn:sg2-weak}, but since
we are only able to prove that the velocity $u_t$ is
locally in $L^1$, Equation \eqref{eqn:sg2-weak} makes sense only with
compactly supported $\psi$.
On the other hand, as we shall explain in Remark \ref{flow},
we will be able to prove that there exists a measure preserving Lagrangian
flow $F_t: \Omega \to \Omega$ associated to $u_t$, and such existence
result can be interpreted as a very weak formulation of the constraint
$u_t \cdot \nu_{\Omega}=0$.

As pointed out to us by Cullen, this weak boundary condition is actually
very natural: indeed, the classical  boundary condition would prevent the
formation of ``frontal singularities'' (which are physically expected to
occur), i.e. the fluid initially at the boundary would not be able to move
into the interior of the fluid, while this is allowed by our weak version
of the boundary condition.
\end{remark}

We can now state our main result.

\begin{theorem}\label{thm:main}
Let $\Omega\subseteq \R^3$ be a convex bounded open set,
and let $\L_\Omega$ be the normalized  Lebesgue measure restricted to $\Omega$,
that is $\L_{\Omega}(\Omega)=1$. Let $\rho_0$ be a probability density on $\R^3$ such that $\rho_0 \in L^{\infty}(\R^3)$, $1/\rho_0 \in L^\infty_{\rm loc}(\R^3)$ and
$$
\limsup_{|x| \to\infty} \left( \rho_0(x) |x|^K \right) <\infty
$$
for some $K>4$.  Let $\rho_t$ be a solution of \eqref{eqn:dualsystem} given by
Theorem~\ref{thm:dualeq}, 
$P_t^*:\R^3\to \R$ the unique convex
function such that
\[
P_t^*(0)=0 \quad \text{and}\quad(\nabla P_t^*)_\sharp(\rho_t\L^3)=\L_\Omega,
\]
and let $P_t:\R^3\to\R$ be its convex conjugate.

Then the vector field $u_t$ in \eqref{defn:ut} is well defined, and the couple $(P_t,u_t)$ is a weak Eulerian solution of
\eqref{eqn:SGsystem2} in the sense of
Definition~\ref{sg-weak-eul}.
\end{theorem}

\begin{remark}\label{flow}Following Cullen and Feldman one can give also a notion of Lagrangian
solution of the semigeostrophic equation. More precisely they show the
existence of a measure preserving flow $F_t :\Omega \to \Omega$ which solves a sort of Lagrangian version of
\eqref{eqn:SGsystem1} (see \cite{cufe} and \cite[Section 5]{acdf} for a more precise discussion).
Actually the flow they constructed has the explicit expression $F_t=\nabla P^*_t \circ G_t\circ \nabla P_0$,
where $G_t$ is the regular Lagrangian flow associated to the $BV$ vector field
$U_t=J(x-\nabla P^*_t)$, in the sense of  Ambrosio, Di Perna and Lions (see \cite{AM1, AM2, DL}).
In \cite[Section 5]{acdf} we showed, in the two dimensional periodic setting, that
for almost every $x$ the map $t \mapsto F_t(x)$ is absolutely continuous with derivative given
by $u_t(F_t(x))$. The proof of this fact can be almost verbatim extended to our contest,
showing that, for almost every $x\in \Omega$,  $t \mapsto F_t(x)$ is
locally absolutely continuous in $[0,\infty)$ with derivative given by $u_t(F_t(x))$. We leave the proof of this fact to the interested reader. Finally we remark that the uniqueness of such a flow (both according to the definition given in \cite{cufe} or in \cite{acdf}) is unknown.

\end{remark}

\smallskip
\noindent {\bf Acknowledgement.} L.A., G.D.P., and A.F.
acknowledge the support of the ERC ADG GeMeThNES.
A.F. was also supported by the NSF Grant DMS-0969962. M.C. and  G.D.P. also want to acknowledge the hospitality
of the University of Texas at Austin, where part of this work has been done.

\section{Regularity of optimal transport maps between convex sets of $\R^3$}
\label{sect:MA}

Throughout this paper, $\Omega\subseteq \R^3$ is a bounded convex open set, $\diam>0$
is fixed in such a way that 
$ \overline{\Omega} \subset B(0,\diam)$,
and $\L_\Omega$ denotes the normalized Lebesgue measure restricted  to $\Omega$.

In this section we recall some regularity results for optimal transport maps in $\R^3$ needed in the paper.

\begin{theorem}[Space regularity of optimal maps between convex sets]\label{thm:transport reg}
Let $\Omega_0$, $\Omega_1$ be open sets of $\R^3$, with $\Omega_1$ bounded and convex.
Let $\mu=\rho\L^3$ and $\nu=\sigma\L^3$
be probability densities such that $\mu( \Omega_0)=1$, $\nu (\Omega_1)=1$. Assume that the density $\rho$  is locally bounded both from above and from below in $\Omega_0$, namely that for every compact set $K\subset \Omega$ there exist $\lambda_0=\lambda_0 (K)$ and $\Lambda_0=\Lambda_0(K)$ satisfying
\[
0<\lambda_0\le \rho(x)\le \Lambda_0 \quad \forall\; x \in K.
\]
Futhermore, suppose that $\lambda_1 \leq \sigma(x)\leq \Lambda_1$ in $\Omega_1$.
Then the following properties hold true.
 \begin{enumerate}
 \item[(i)]  There exists a unique optimal transport map between $\mu$ and $\nu$, namely a unique (up to an additive constant) convex function $P^*:  \Omega_0\to \R$ such that $(\nabla P^*)_\sharp \mu = \nu$. Moreover $P^*$ is a strictly convex Alexandrov solution of
$$\det\nabla^2 P^*(x) = f(x),\qquad\text{with }f(x)=\frac{\rho(x)}{\sigma(\nabla P^*(x))}.$$
\item[(ii)]\label{thm:tr2} $P^*\in W_{\rm loc}^{2,1}(\Omega_0)\cap C^{1,\beta}_{\rm loc} (\Omega_0)$. More precisely, if $\Omega \Subset\Omega_0$ is an open set and $0<\lambda\leq \rho(x) \leq \Lambda<\infty$ in $\Omega$, then for any $k \in \N$ there exist  constants $C_1=C_1(k, \Omega,\Omega_1, \lambda, \Lambda, \lambda_1, \Lambda_1)$, $\beta=\beta  (\lambda, \Lambda, \lambda_1, \Lambda_1)$, and $C_2=C_2( \Omega,\Omega_1, \lambda, \Lambda, \lambda_1, \Lambda_1)$ such that
$$
\int_{\Omega}|\nabla^2 P^*| \log_+^k |\nabla^2 P^* | \,dx \leq C_1,
$$
and 
\[
\|P^*\|_{C^{1,\beta}(\Omega)}\le C_2.
\]

\item[(iii)] Let us also assume that $\Omega_0$, $\Omega_1$ are bounded and uniformly convex,
$ \partial\Omega_0, \partial\Omega_1 \in C^{2,1}$, $\rho\in C^{1,1}(\Omega_0)$, $\sigma \in C^{1,1}(\Omega_1)$,
and $\lambda_0 \leq \rho(x) \leq \Lambda_0$ in $\Omega_0$. Then
$$P^*\in C^{3,\alpha}(\Omega_0) \cap  C^{2,\alpha}(\overline{\Omega}_0) \qquad \forall \; \alpha \in (0,1),$$
and there exists a constant
$C$ which depends only on $\alpha, \Omega_0, \Omega_1,\lambda_0,\lambda_1, \|\rho\|_{C^{1,1}},
\|\sigma\|_{C^{1,1}}$  such that
$$\| P^*\|_{C^{3,\alpha}(\Omega_0)}\leq C \qquad \mbox{and} \qquad \| P^*\|_{C^{2,\alpha}(\overline{\Omega}_0)}\leq C.$$
Moreover, there exist positive constants $c_1$ and $c_2$ and $\kappa$,
depending only on $\lambda_0, \lambda_1, \|\rho\|_{C^{0,\alpha}}$,
and $\|\sigma\|_{C^{0,\alpha}}$, such that
$$c_1 Id \leq \nabla^2  P^*(x) \leq c_2 Id\qquad\forall \,x\in\Omega_0$$
and
\[
\nu_{\Omega_1}(\nabla P^*(x))\cdot \nu_{\Omega_0}(x) \ge \kappa\quad \forall \,x\in\partial \Omega_0. 
\]
\end{enumerate}
\end{theorem}
The first statement is standard optimal transport theory, see \cite{CA4,MC}, except for the fact that we are not assuming that the second moment of $\mu$ is finite, thus the classical Wasserstein distance from $\mu$ and $\nu$ can be infinite. Nevertheless the existence of an ``optimal'' map is provided by \cite{MC}. The $W^{2,1} $ part of the second statement follows from a recent regularity result about solutions of the Monge-Amp\`ere equation \cite{DepFi}, while the $C^{1,\beta}$ regularity was proven by Caffarelli in \cite{CA1,CA2, CA4}.
The regularity up to the boundary and the oblique derivative condition of the third statement have been proven by Caffarelli \cite{CA3} and Urbas \cite{Ur}.
\begin{remark}\label{rmk:set-conv}
By compactness and a standard contradiction argument,
the constants $C_1$ and $C_2$ in the statement (ii) of the previous
theorem remain uniformly bounded if $\Omega_1$ varies in a compact class (with respect,
for instance, to the Hausdorff distance) of convex sets. In particular, let  $\Omega_1^n$ be a
sequence of open convex sets which converges to $\Omega_1$ with respect to the Hausdorff distance
and  $\sigma_n$ a sequence of densities supported on $\Omega_1^n$ with $\lambda_1 \leq \sigma_n
\leq \Lambda_1$ on $\Omega_1^n$ which converge to $\sigma$ in $L^1(\R^3)$. Then the estimates in Theorem
\ref{thm:transport reg}(ii) hold true with constants independent of $n$.
\end{remark}

\begin{remark}\label{exp1}
As already mentioned in the introduction, in statement (ii) the optimal regularity is that for every $\Omega\Subset \Omega_0$ and  $0<\lambda\leq \rho(x) \leq \Lambda<\infty$ in $\Omega$, there exist $\gamma (\lambda, \Lambda, \lambda_1, \Lambda_1)>1$ and $C(\Omega, \Omega_1, \lambda, \Lambda, \lambda_1, \Lambda_1)$, such that
\[
\int_{\Omega} |D^2 u|^{\gamma} \le C.
\]
However, as explained in Remark \ref{exp2},
this improvement does not give any advantage.
\end{remark}

\section{The dual problem and the regularity of the velocity field}\label{sect:dual}

In this section we recall some properties of solutions of
\eqref{eqn:dualsystem}, and we show the $L^1$ integrability
of the velocity field $u_t$ defined in \eqref{defn:ut}.

We have the following result whose proof follows adapting the argument of \cite{bebr, CG}, where compactly supported initial data are
considered. Since the velocity $U_t$ has at most linear growth, the speed of propagation is locally finite and the proof readily extends to
general probability densities.

\begin{theorem}[Existence of solutions of \eqref{eqn:dualsystem}]\label{thm:dualeq}  Let $P_0:\R^3\to\R$ be a convex function such that $(\nabla P_0)_\sharp\L_{\Omega}\ll\L^3$.
Then there exist convex functions $P_t,P_t^* :\R^3\to\R$ such that $(\nabla P_t)_\sharp
\L_\Omega=\rho_t\L^3$,  $(\nabla P_t^*)_\sharp
\rho_t=\L_{\Omega}$, $U_t(x)=J(x-\nabla
P_t^*(x))$, and $\rho_t$ is a distributional solution to
\eqref{eqn:dualsystem}, namely
\begin{equation}\label{eqn:sg-dual-weak}
 \int \int_{\R^3}
\Big\{ \partial_t \varphi_t(x) + \nabla \varphi_t(x) \cdot U_t(x)
\Big\} \rho_t(x)\, dx \, dt + \int_{\R^3} \varphi_0(x) \rho_0(x)
\, dx = 0
\end{equation}
for every $\varphi \in C^{\infty}_c(\R^3\times [0,\infty))$.

Moreover, the following regularity properties hold:
\begin{enumerate}
\item[(i)] $\rho_t\L^3 \in C([0,\infty), \mathcal P_w (\R^3))$, where
$\mathcal P_w(\R^3)$ is the space of probability measures  endowed with the \emph{weak} topology induced by the duality
with $C_0(\R^3)$;
\item[(ii)] $P^*_t-P^*_t(0) \in L^\infty_{\rm loc}([0,\infty),W_{\rm loc}^{1,\infty}(\R^3))\cap C([0,\infty),W_{\rm loc}^{1,r}(\R^3))$
for every $r\in [1,\infty)$; 
\item[(iii)] $|U_t(x)| \leq |x| +\diam$ for almost every $x\in \R^3$, for all $t\geq 0$.
\end{enumerate}
\end{theorem}

Observe that, by Theorem \ref{thm:dualeq}(ii), $t \mapsto \rho_t\L^3$ is weakly continuous,
so $\rho_t$ is a well-defined function \textit{for every} $t\geq 0$.
Further regularity properties of $P_t$ and $P_t^*$ with respect to
time will be proven in Proposition~\ref{prop:est}.

In the proof of Theorem~\ref{thm:main} we will need to test  with
functions which are merely $W^{1,1}$ with compact support. This is made possible by a simple approximation argument which we leave to the reader, see \cite[Lemma 3.2]{acdf}.

\begin{lemma}\label{rmk:scontr}
Let $\rho_t$ and $P_t$ be as in Theorem~\ref{thm:dualeq}. Then
\eqref{eqn:sg-dual-weak} holds for every $\varphi \in
W^{1,1}(\R^3\times [0,\infty))$ which is compactly supported in time and space, where now 
$\varphi_0(x)$ has to be understood in the sense of
traces.
\end{lemma}

\begin{lemma}[Space-time regularity of transport]\label{lemma:MAlin}
Let $\Omega \subseteq \R^3$ be a uniformly convex bounded domain
with $\partial\Omega \in C^{2,1}$, let  $R>0$, and consider
$\rho\in C^\infty(\overline{B(0,R)} \times [0,\infty))$
and $U\in
C^\infty_c(B(0,R) \times [0,\infty);\R^3)$ satisfying
$$\partial_t \rho_t +\nabla \cdot(U_t\rho_t) = 0 \quad\quad\text{in $B(0,R)\times [0,\infty).$}$$
Assume that $\int_{B(0,R)}\rho_0\,dx=1$, and that for every $T>0$ there exist $\lambda_T$ and $\Lambda_T$ such that
$$0<\lambda_T \leq \rho_t(x) \leq \Lambda_T<\infty \quad\quad \forall\, (x,t)\in B(0,R)\times [0,T].$$
Consider the
 convex conjugate maps $P_t$ and $P_t^*$ such that 
 $(\nabla P_t)_\sharp \L_{\Omega}=\rho_t$ and
 $(\nabla P_t^*)_\sharp
\rho_t=\L_{\Omega}$ (unique up to additive constants in $\Omega$ and $B(0,R)$ respectively). Then:
\begin{enumerate}
\item[(i)] $P^*_t -\negint_{B(0,R)}P_t^* \in {\rm Lip_{\rm loc}}
([0,\infty);C^{2, \alpha}(\overline{B(0,R)}))$.
\item[(ii)] The following linearized Monge-Amp\`ere equation holds for every $t\in [0,\infty)$:
 \begin{equation}
 \begin{cases}
 \nabla \cdot \bigl(\rho_t (\nabla^2P_t^*)^{-1}\partial_t \nabla P_t^*\bigr)  = - \nabla \cdot(\rho_t U_t)
\quad\quad &\text {in $ B(0,R)$} \\
\rho_t (\nabla^2P_t^*)^{-1}\partial_t \nabla P_t^* \cdot \nu= 0 \quad\quad &\text {on $ \partial B(0,R)$.}
\end{cases}
 \label{ts:MAlin}
 \end{equation}
\end{enumerate}
\end{lemma}
\begin{proof}
Observe that because $\rho_t$ solves a continuity equation
with a smooth compactly supported vector field, $\int_{B(0,R)}\rho_t\,dx=1$ for all $t$.

Let us fix $T>0$.
From the regularity theory for the Monge-Amp\'ere equation (Theorem~\ref{thm:transport reg} applied to $P_t$ and $P_t^*$) we obtain that $P_t\in
C^{3,\alpha}(\Omega) \cap C^{2,\alpha}(\overline{\Omega})$ and  $P_t^*\in
C^{3,\alpha}(B(0,R)) \cap C^{2,\alpha}(\overline{B(0,R)})$ for every $\alpha\in (0,1)$, uniformly for
$t \in [0,T]$, and there exist constants
$c_1,\,c_2>0$ such that
\begin{equation}
c_1 Id \leq \nabla^2 P_t^*(x) \leq c_2 Id\qquad\forall\, (x,t)\in B(0,R)\times[0,T].
\label{est:d2}
\end{equation}
Let $h\in C^{2,1}(\R^3)$ be a convex function such that $\Omega = \{y: h(y)< 0\}$ and $|\nabla h (y)| = 1$ on 
$\partial \Omega$, so that $\nabla h(y)=\nu_\Omega(y)$.
Since $\nabla P_t^*\in C^{1,\alpha}(\overline{B(0,R)})$, it is a diffeomorphism onto its image, we have
\begin{equation}
\label{eqn:dir-bound}
h(\nabla P_t^*(x)) = 0 \qquad \forall\, (x,t) \in \partial B(0,R) \times [0,T].
\end{equation}

To prove (i) we need to investigate the time regularity of $P_t^*- \negint_{B(0,R)} P_t^*$.

Possibly adding a time dependent constant to $P_t$, we can assume
without loss of generality that $\int_{B(0,R)} P_t^*=0$ for all $t$.
By the condition $(\nabla P_t^*)_\sharp \rho_t=\L_{\Omega}$ we get that for
any $0\le s,t\le T$ and $x\in B(0,R)$ it holds
\begin{equation} \label{eqn:rho-dif}
\begin{split}
\frac{ \rho_s(x) - \rho_t(x)}{s-t} &= \frac{ \det( \nabla^2P_s^*(x)) -\det( \nabla^2P_t^*(x))}{s-t}\\
&= \sum_{i,j=1}^3 \biggl(\int_0^1 \frac {\partial\det}{\partial \xi_{ij}} (
\tau \nabla^2  P_s^*(x)+ (1-\tau) \nabla^2 P_t^*(x)) \, d\tau\biggr) \,
\frac{ \partial _{ij} P_s^*(x) -\partial _{ij}P_t^*(x)}{s-t}.
\end{split}
\end{equation}
Moreover, from \eqref{eqn:dir-bound} we obtain that on $\partial B(0,R)$ 
\begin{equation} \label{eqn:bound-dif}
\begin{split}
 0 &= \frac{ h(\nabla P_s^*(x)) - h(\nabla P_t^*(x))}{s-t}\\
 & = \int_0^1 \nabla h (
\tau \nabla P_s^*(x)+ (1-\tau) \nabla P_t^*(x)) \, d\tau \,\cdot
\frac{ \nabla P_s^*(x) - \nabla P_t^*(x)}{s-t}.
\end{split}
\end{equation}
Now, given a matrix $A=(\xi_{ij})$, we denote by
$M(A)$ the cofactor matrix of $A$. We recall that
\begin{equation}\label{eqn:det-deriv}
 \frac {\partial\det (A)}{\partial \xi_{ij}}  = M_{ij}(A),
\end{equation}
and if $A$ is invertible then $M(A)$
satisfies the identity
\begin{equation}
 M(A)= \det(A) \, A^{-1}.
\label{eqn:cof}
\end{equation}
Moreover, if $A$ is symmetric and satisfies $c_1 Id \leq A\leq c_2
Id$ for some positive constants $c_1,\,c_2$, then
\begin{equation}\label{eqn:cofactor elliptic}
 c_1^2 Id\leq M(A) \leq  c_2^2 Id.
\end{equation}
Hence, from \eqref{eqn:rho-dif}, \eqref{eqn:det-deriv}, \eqref{est:d2} and \eqref{eqn:cofactor
elliptic} it follows that
\begin{equation}\label{eqn:dg-reg}
\frac{ \rho_s - \rho_t}{s-t} = \sum_{i,j=1}^3 \biggl(\int_0^1M_{ij}( \tau
\nabla^2  P_s^*+ (1-\tau) \nabla^2 P_t^*) \, d\tau\biggr) \, \partial _{ij}\biggl(\frac{
P_s^* -P_t^*}{s-t}\biggr),
 \end{equation}
with
\[
c_1^2 Id \le \int_0^1 M_{ij} ( \tau \nabla^2 P_s^*+ (1-\tau)
\nabla^2 P_t^*) \, d\tau\le c_2^2 Id.
\]
Also, from Theorem \ref{thm:transport reg}(iii) the oblique derivative condition holds, namely there exists $\kappa>0$ such that
$$ \nabla h (\nabla P_t^*(x)) \cdot \nu_{B(0,R)}(x) \geq \kappa \qquad \forall \,x\in \partial B(0,R).$$
Thus, since 
$$ \lim_{s\to t} \int_0^1 \nabla h ( \tau \nabla P_s^*(x)+ (1-\tau) \nabla P_t^*(x)) \, d\tau = \nabla h (\nabla P_t^*(x)) $$
uniformly in $t$ and $x$, we have that
$$ \int_0^1 \nabla h ( \tau \nabla P_s^*(x)+ (1-\tau) \nabla P_t^*(x)) \, d\tau \cdot \nu_{B(0,R)}(x) \geq \frac{\kappa}{2}$$
for $s-t$ small enough.

Hence, from the regularity theory for the oblique derivative problem \cite[Theorem 6.30]{GT} we obtain that for any $\alpha \in (0,1)$ there exists a constant $C$ depending only on $\Omega$, $T$, $\alpha$, 
$\| (\rho_s - \rho_t) /(s-t) \|_{C^{0,\alpha}(B(0,R))}$, such that
\begin{equation*}
  \left\|\frac{ P_s^*(x) -P_t^*(x)}{s-t}\right\|_{C^{2,\alpha}(\overline{B(0,R)})}\leq C.
  \label{eqn:pot-est}
\end{equation*}
Since $\partial_t \rho_t \in L^\infty([0,T], C^{0,\alpha}(B(0,R)))$, this proves point (i) in the statement. To prove the second part, we let $s \to t$ in
\eqref{eqn:dg-reg} to obtain
\begin{equation} \label{eqn:non-var}
\partial_t \rho_t= \sum_{i,j=1}^{3}M_{ij}( \nabla^2 P_t^*(x)) \,  \partial_t \partial_{ij} P_t^*(x).
\end{equation}
Taking into account the continuity equation and the well-known divergence-free property of the cofactor matrix
\[
\sum_{i=1}^3 \partial_i M_{ij}( \nabla^2 P_t^*(x)) = 0,\qquad j=1,2,3,
\]
we can rewrite \eqref{eqn:non-var}  as
\[
-\nabla \cdot(U_t \rho_t) = \sum_{i,j=1}^{3}\partial_i\bigl(M_{ij}(
\nabla^2 P_t^*(x)) \,  \partial_t \partial_{j} P_t^*(x)\bigr).
\]
Hence, using \eqref{eqn:cof} and the Monge-Amp\'ere equation $\det(\nabla^2 P_t^*)=\rho_t$,
we get equation \eqref{ts:MAlin}.

In order to obtain the boundary condition in  \eqref{ts:MAlin}, we take to the limit as $s\to t$ in \eqref{eqn:bound-dif} to get
\begin{equation}\label{eqn:norm-lim}
\nabla h (\nabla P_t^*(x)) \cdot \partial_t \nabla P_t^*(x)=0.
\end{equation}
Since $h$ satisfies $\Omega = \{y: h(y)< 0\}$ and $\nabla P_t^*$ maps $B(0,R)$ in $\Omega$, we have that $B(0,R) = \{y: h(\nabla P_t^*(y))< 0\}$. Hence $\nu_{B(0,R)}(x)$ is proportional to $\nabla [h \circ \nabla P_t^*](x) = \nabla^2 P_t^*(x) \nabla h(\nabla P_t^*(x))$, which implies that the exterior normal to $\Omega$ at point $\nabla P_t^*(x)$, which is $\nabla h (\nabla P_t^*(x))$,
 is collinear with $\rho_t (\nabla^2P_t^*)^{-1} \nu_{B(0,R)}$. Hence from \eqref{eqn:norm-lim} it follows that
 $$ \rho_t (\nabla^2P_t^*)^{-1}\nu_{B(0,R)} \cdot \partial_t \nabla P_t^*= 0,$$
as desired.
\end{proof}

\begin{lemma}[Decay estimates on $\rho_t$]\label{lemma:dec-est}
Let $v_t: \R^3 \times [0,\infty) \to \R^3$ be a $C^\infty$ velocity field and suppose that 
$$\sup_{x,t}|\nabla \cdot v_t(x)| \leq N,\qquad
 |v_t(x)| \leq A |x| +D \quad \forall\, (x,t)\in \R^3\times [0,\infty)$$
 for suitable constants $N,\,A, \,D$.
Let $\rho_0$ be a probability density, and let $\rho_t$ be the solution of the continuity equation
 \begin{equation}\label{eqn:cont}
 \partial_t \rho_t+\nabla\cdot (v_t\rho_t)=0\qquad\text{in $\R^3\times (0,\infty)$}
 \end{equation}
starting from $\rho_0$. Then:
\begin{itemize}
\item[(i)] For every $r>0$ and $t\in[0,\infty)$ it holds
\begin{equation}\label{hp:bound-above}
\|\rho_t\|_\infty\leq e^{Nt}\|\rho_0\|_\infty,
\end{equation}
\begin{equation}\label{hp:bound-below}
\rho_t(x) \ge e^{-Nt}\inf\Big\{\rho_0(y):\ y\in B\Big(0,re^{At}+D \frac{e^{At}-1}{A}\Big)\Big\}  \qquad \forall \,x \in B(0,r).
\end{equation}
\item[(ii)] Let us assume that there exist $d_0\in [0,\infty)$ and $M\in [0,\infty)$ such that
\begin{equation}\label{hp:growth-c}
\rho_0(x) \leq \frac{d_0}{|x|^K} \qquad \mbox{whenever} \quad |x|\geq M.
\end{equation}
Then for every $t\in[0,\infty)$ we have that 
\begin{equation}\label{ts:growth-c}\rho_t(x) \leq\frac{d_0 2^{K} e^{(N+AK)t}}{|x|^K} \qquad \text{whenever}\quad|x|\geq 2Me^{At}+2D \frac{e^{At}-1}{A}.\end{equation}
\item[(iii)]\label{st:dec-1} 
Let us assume that there  exists $R>0$ such that
$\rho_0$ is smooth in $\overline{B(0,R)}$, vanishes outside $\overline{B(0,R)}$, and that
$v_t$ is compactly supported inside $B(0,R)$ for all $t\geq 0$.
Then $\rho_t$ is smooth inside $\overline{B(0,R)}$ and vanishes outside $\overline{B(0,R)}$ for
all $t\geq 0$. Moreover
if $0<\lambda\le\rho_0\le\Lambda<\infty$ inside $B(0,R)$, then
\begin{equation}\label{ts:supp-rho}
\lambda e^{-tN}\leq \rho_t\leq\Lambda  e^{tN}\quad\text{inside $B(0,R)$ for all  $t\geq 0$.}
\end{equation}
\end{itemize}
\end{lemma}
\begin{proof}
Let $X_t(x)\in C^\infty(\R^3\times [0,\infty))$ be the flow associated to the velocity field $v_t$, namely the solution to
\begin{equation}\label{eqn:vt-flow}
\begin{cases}
\frac{d}{dt} X_t(x) = v_t(X_t(x)) \\
X_0(x) = x.
\end{cases}
\end{equation}
For every $t\geq 0$ the map $t\mapsto X_t(x)$ is invertible in $\R^3$, with inverse denoted by $X_t^{-1}$.

The solution to the continuity equation \eqref{eqn:cont} is given by $\rho_t = {X_t}_\sharp \rho_0$,
and from the well-known theory of characteristics it can be written explicitly using the flow:
\begin{equation}\label{eqn:rhot-exp}\rho_t(x)= \rho_0(X_t^{-1}(x)) e^{\int_0^t \nabla \cdot v_s(X_{ s}(X_t^{-1}(x))) \, ds} \qquad \forall\, (x,t) \in \R^3 \times [0,\infty).
\end{equation}

Since the divergence is bounded, we therefore obtain
\begin{equation}\label{eqn:rhot-espl}
\rho_0(X_t^{-1}(x)) e^{-Nt} \leq \rho_t(x) \leq \rho_0(X_t^{-1}(x)) e^{Nt}
\end{equation}
Now we deduce the statements of the lemma from the properties of the flow $X_t$.

(i) From \eqref{eqn:rhot-espl} we have that 
$$ \rho_t(x) \leq  e^{Nt} \rho_0(X_t^{-1}(x)) \leq  e^{Nt} \sup_{x\in \R^3} \rho_0(x) ,$$
which proves \eqref{hp:bound-above}. From the equation \eqref{eqn:vt-flow} we obtain 
$$\Bigl|\frac{d}{dt} |X_t(x)|\Bigr| \leq |\partial_t X_t(x)| \leq  A| X_t(x)|+D$$
which can be rewritten as
\begin{equation}\label{eq:timeder}
- A  |X_t(x)| - D \leq \frac{d}{dt} |X_t(x)| \leq A  |X_t(x)| +D.
\end{equation}
From the first inequality we get
$$|X_t(x)| \geq |x|e^{-At} - D \frac{1-e^{-At}}{A},$$
which implies
$$|x|e^{At}+D \frac{e^{At}-1}{A}\geq  |X_t^{-1}(x)|,$$
or equivalently
\begin{equation}\label{est:x-1}
X_t^{-1}\bigl(\{|x|\leq r\}\bigr) \subseteq \Bigl\{ |x|\leq re^{At}+D \frac{e^{At}-1}{A}\Bigr\}.
\end{equation}
Hence from \eqref{eqn:rhot-espl} and \eqref{est:x-1} we obtain that, for every $x\in B(0,r)$,
\begin{eqnarray*}
\rho_t(x) &\geq&  e^{-Nt} \rho_0(X_t^{-1}(x))\\
&\geq& e^{-Nt} \inf\{ \rho_0(y) : y\in {X_t^{-1}(B_r(0))}\}\\
&\geq& e^{-Nt}  \inf\Bigl\{ \rho_0(y) : { |y|\leq re^{At}+D \frac{e^{At}-1}{A}}\Bigr\},
\end{eqnarray*}
which proves \eqref{hp:bound-below}.

(ii) From the second inequality in \eqref{eq:timeder}, we infer 
$$|X_t(x)| \leq |x|e^{At}+D \frac{e^{At}-1}{A},$$
which implies
\begin{equation}\label{est:x2}
|x|\leq  |X_t^{-1}(x)|e^{At}+D \frac{e^{At}-1}{A}.
\end{equation}
Thus, if $|x|\geq 2Me^{At}+2D \frac{e^{At}-1}{A}$, we easily deduce
from \eqref{est:x2} that $|X_t^{-1}(x)|\geq M+|x|e^{-At}/2$,
so by \eqref{hp:growth-c}
$$ \rho_t(x)\le  e^{Nt} \rho_0(X_t^{-1}(x)) \leq \frac{d_0 e^{Nt}}{|X_t^{-1}(x)|^K} \leq
\frac{d_0 2^{K} e^{(N+AK)t}}{|x|^K}, $$
which proves \eqref{ts:growth-c}.

(iii) { If $v_t = 0$ in a neighborhood of $\partial B(0, R)$ it can be easily verified that the flow maps $X_t:\R^3 \to \R^3$ leave both $B(0,R)$ and its complement invariant. Moreover
the smoothness of $v_t$ implies that also $X_t$ is smooth.
Therefore all the properties of $\rho_t$ follow directly from \eqref{eqn:rhot-exp}. }
\end{proof}

We are now ready to prove the regularity of $\nabla P_t^*$.

\begin{proposition}[Time regularity of optimal maps]
\label{prop:est} Let $\Omega\subseteq \R^3$ be a bounded, convex, open set and let $\diam$ be such that $ \overline{\Omega} \subset B(0,\diam)$. Let $\rho_t$ and $P_t$ be as in Theorem~\ref{thm:dualeq}, in addition let us assume that there exist $K>4$, $M\geq 0$ and $c_0>0$ such that
\begin{equation}\label{hp:dec}
\rho_0(x) \leq \frac{c_0}{|x|^K} \qquad \text{whenever $|x|\geq M$.}
\end{equation}
Then
$\nabla P_t^*\in W^{1,1}_{\rm loc}(\R^3\times [0,\infty);\R^3)$. Moreover  for
every $k\in\n$ and $T>0$ there exists a constant $C = C(k, T, M, c_0, \| \rho_0\|_\infty, \diam)$ such that, for almost every $t \in [0,T]$ it holds
\begin{equation}
\label{ts:loep}
\int_{B(0, r)} \rho_t |\partial_t \nabla P_t^*| \log^k_+ (|\partial_t \nabla P_t^*|) \, dx
\leq 2^{3(k-1)} \int_{B(0, r)}\rho_t |\nabla^2 P_t^*|
\log^{2k}_+(|\nabla^2 P_t^*|) \, dx+ C\qquad\forall \,r>0.
\end{equation}
\end{proposition}

\begin{proof} \emph{Step 1: The smooth case.} In the first part of the proof we assume that $\Omega$ is a convex smooth domain, 
and, besides \eqref{hp:dec}, that for some $R>0$ the following additional properties hold:
\begin{eqnarray}
&&\rho_t\in C^{\infty}(\overline{B(0,R)} \times \R ),\,\,U_t\in C^{\infty}_c(B(0,R) \times \R;\R^3 ),\,\, |\nabla \cdot U_t|\le N
\label{eqn:loep-reg1}\\
&&\lambda 1_{B(0,R)}(x)  \le \rho_0(x)\le
\Lambda 1_{B(0,R)}(x)\qquad\forall \,x\in\R^3\label{eqn:loep-reg2}  ,\\
&&\partial_t
\rho_t+\nabla\cdot (U_t\rho_t)=0 \qquad\text{in $\R^3\times [0,\infty)$}\,,\label{eqn:loep-reg4}\\
&&
(\nabla P^*_t)_\sharp \rho_t=\L_\Omega ,\label{eqn:loep-reg5}\\
&& |U_t(x)| \leq |x|+\diam\label{eqn:loep-reg6}
\end{eqnarray}
for some constants $N, \,\lambda, \,\Lambda$, and we prove that \eqref{ts:loep} holds for every $t\in [0,T]$. 
Notice that in this step we do not assume any coupling between the velocity $U_t$ and the transport map $\nabla P^*_t$.
In the second step we prove the general case through an approximation argument.

Let us assume that the regularity assumptions \eqref{eqn:loep-reg1}
through \eqref{eqn:loep-reg6} hold. By Lemma~\ref{lemma:dec-est} we infer that,
for any $T>0$, there exist positive constants
$\l_T,\Lambda_T, c_T,M_T$, with $M_T \geq 1$,  such that
\begin{eqnarray}
&&\lambda_T 1_{B(0,R)}(x) \le  \rho_t(x)\le
\Lambda_T 1_{B(0,R)}(x)\label{eqn:evloep-reg1}  ,\\
&&\rho_t(x)\le
\frac{c_T}{|x|^K}\quad \text{ for $|x|\ge M_T,\qquad$ for all $ t\in [0,T]$. }\label{eqn:evloep-reg2}
\end{eqnarray}
By Lemma~\ref{lemma:MAlin} we have that $\partial_t P_t^* \in
C^{2}( \overline{B(0,R)})$, and it solves
 \begin{equation}
 \begin{cases}
 \nabla \cdot \bigl(\rho_t (\nabla^2P_t^*)^{-1}\partial_t \nabla P_t^*\bigr)  = - \nabla \cdot(\rho_t U_t)
\quad\quad &\text {in $ B(0,R)$} \\
\rho_t (\nabla^2P_t^*)^{-1}\partial_t \nabla P_t^* \cdot \nu= 0 \quad\quad &\text {in $ \partial B(0,R)$.}
\end{cases}
\label{ts:MAlin1}
\end{equation}
Multiplying \eqref{ts:MAlin1} by $\partial_t P_t^*$ and integrating
by parts, we get
\begin{equation}
\label{eqn:ma-div}
\begin{split}
 \int_{B(0,R)} \rho_t | (\nabla^2 P_t^*)^{-1/2}\partial_t \nabla P_t^*|^2\, dx
&=\int_{B(0,R)}  \rho_t \partial_t \nabla P_t^*\cdot (\nabla^2 P_t^*)^{-1} \partial_t \nabla P_t^* \,dx\\
& = -\int_{B(0,R)}  \rho_t  \partial_t \nabla P_t^* \cdot U_t \, dx.
\end{split}
\end{equation}
(Notice that, thanks to the boundary condition in \eqref{ts:MAlin1}, we do not have any boundary term in \eqref{eqn:ma-div}.)
From Cauchy-Schwartz inequality it follows that the right-hand side of \eqref{eqn:ma-div}
can be rewritten and estimated by
\begin{equation}
\begin{split}\label{est:ma}
&-\int_{B(0,R)}  \rho_t  \partial_t \nabla P_t^*\cdot (\nabla^2 P_t^*)^{-1/2} (\nabla^2 P_t^*)^{1/2} U_t\,dx\\
& \leq \left( \int_{B(0,R)} \rho_t | (\nabla^2 P_t^*)^{-1/2}\partial_t
\nabla P_t^*|^2\, dx\right)^{1/2}  \left( \int_{B(0,R)}  \rho_t
|(\nabla^2 P_t^*)^{1/2}  U_t|^2\, dx \right)^{1/2}.
 \end{split}
\end{equation}
Moreover, the second term in the right-hand side of \eqref{est:ma} is controlled by
\begin{equation}\label{est:v}
\int_{B(0,R)}  \rho_t U_t \cdot \nabla^2 P_t^* U_t\,dx \leq
\max_{B(0,R)}\left( \rho_t^{1/2} |U_t|^2 \right) \int_{B(0,R)} \rho_t^{1/2} |\nabla^2 P_t^*| \, dx.
\end{equation}
Hence from \eqref{eqn:ma-div}, (\ref{est:ma}), and \eqref{est:v} we obtain
\begin{equation}\label{est:glob}
 \int_{B(0,R)} \rho_t | (\nabla^2 P_t^*)^{-1/2}\partial_t \nabla P_t^*|^2\, dx \leq 
 \max_{B(0,R)}\left( \rho_t^{1/2} |U_t|^2 \right) \int_{B(0,R)} \rho_t^{1/2} |\nabla^2 P_t^*| \, dx.
\end{equation}

>From \eqref{eqn:loep-reg6},
\eqref{eqn:evloep-reg1}, and \eqref{eqn:evloep-reg2} we estimate the first factor as follows:
\begin{equation}\label{est:rhov-1}
\max_{|x| \leq M_T}\left( \rho_t^{1/2}(x) |U_t(x)|^2 \right)\leq \Lambda_T^{1/2} (M_T+\diam)^2,
\end{equation}
\begin{equation}\label{est:rhov-2}
\max_{ M_T \leq |x|}\left( \rho_t^{1/2}(x) |U_t(x)|^2 \right)\leq
\max_{ M_T \leq |x|}\Bigg \{ \frac{\sqrt{c_T}}{|x|^{K/2}} (|x|+\diam)^2 \Bigg\},
\end{equation}
and the latter term is finite because $M_T\geq 1$ and $K>4$.

In order to estimate the second factor, we observe that
since $\nabla^2 P_t^*$ is a nonnegative matrix the estimate $|\nabla^2 P_t^*| \leq \Delta P_t^*$ holds
(here we are using the operator 
norm on matrices). Hence, by~\eqref{eqn:evloep-reg1}
and \eqref{eqn:evloep-reg2} we obtain
\begin{equation*}
\begin{split}
\int_{B(0,R)} \rho_t^{1/2} |\nabla^2 P_t^*| \, dx & \leq \int_{\{ |x| \leq M_T\}} \rho_t^{1/2} |\nabla^2 P_t^*| \, dx +\int_{\{ |x| > M_T \}} \rho_t^{1/2} |\nabla^2 P_t^*| \, dx \\
& \leq \int_{\{|x| \leq M_T\}} \Lambda_T^{1/2} \Delta P_t^* \, dx +\int_{\{|x| > M_T\}} \frac{\sqrt{c_T}}{|x|^{K/2}} \Delta P_t^* \, dx.
\end{split}
\end{equation*}

The second integral can be rewritten as
$$
\int_{0}^{\infty} \int_{\{|x| > M_T\} \cap \{ {|x|^{-K/2}}>s\}} \Delta P_t^* \, dx \, ds,
$$
which is bounded by
$$
\int_{0}^{[M_T]^{-K/2}} ds \int_{\{ {|x|\leq s^{-2/K}\}}} \Delta P_t^* \, dx.
$$
From the divergence formula, since $|\nabla P_t^* (x)| \leq \diam$ (because $ \nabla P_t^* (x) \in \Omega$ 
for every $x\in \R^3$) and $M_T \geq 1$ (so $[M_T]^{-K/2} \leq 1$) we obtain 
\begin{equation}\label{est:rhoP-1}
\begin{split}
\int_{B(0,R)} \rho_t^{1/2} |\nabla^2 P_t^*| \, dx & 
\leq \Lambda_T^{1/2} \int_{  \{|x| = M_T\}} |\nabla P_t^*| \,d\mathcal H^{2}+ \sqrt{c_T}\int_{0}^{[M_T]^{-K/2}} ds\int_{ \{ {|x|= s^{-2/K}\}}} |\nabla P_t^*| \,\,d\mathcal H^{2}\\
&\leq 4 \pi \Lambda_T^{1/2} M^2_T \diam +  4 \pi \sqrt{c_T}\diam \int_{0}^{1} s^{-4/K} \, ds
\end{split}
\end{equation}
for all $t \in [0,T]$.
Since  $K>4$ the last integral is finite, so the right-hand side is bounded
and we obtain a global-in-space estimate on the left-hand side.

Thus, from (\ref{est:glob}), \eqref{est:rhov-1},  \eqref{est:rhov-2}, and \eqref{est:rhoP-1}, it follows\
that there exists a constant $C_1 = C_1(T,M, c_0, \Lambda, \diam)$ (notice that the constant does not depend on the lower bound on the density) such that 
\begin{equation}\label{est:l2norms}
 \int_{B(0,R)} \rho_t | (\nabla^2 P_t^*)^{-1/2}\partial_t \nabla P_t^*|^2\, dx
 \leq C_1.
 \end{equation}
Applying now the inequality
\begin{equation}
ab \log ^k_+(ab) \leq 2^{k-1} \left[ \left(\frac{k}{e}\right)^k +1 \right] b^2 + 2^{3(k-1)}a^2 \log^{2k}_+ (a) \quad\quad\forall \,(a,b)\in \re^+ \times \re^+,
\label{eqn:dis-num}
\end{equation}
(see \cite[Lemma 3.4]{acdf}) with $a = |(\nabla^2
P_t^*)^{1/2}|$ and $b=  |(\nabla^2 P_t^*)^{-1/2}\partial_t \nabla
{P_t}^*(x)| $ we deduce the existence of a constant $C_2=C_2(k)$ such that
\begin{equation*} \begin{split}
|\partial_t \nabla P_t^*| \log_+^k(|\partial_t \nabla P_t^*|) &\leq
2^{3(k-1)} |(\nabla^2 P_t^*)^{1/2}|^2 \log^{2k}_+(|(\nabla^2
P_t^*)^{1/2}|^2) +C_2
| (\nabla^2 P_t^*)^{-1/2}\partial_t \nabla P_t^*|^2\\
&=2^{3(k-1)} |\nabla^2 P_t^*|\log^{2k}_+(|\nabla^2 P_t^*|) + C_2 |
(\nabla^2 P_t^*)^{-1/2}\partial_t \nabla P_t^*|^2.
\end{split}
\end{equation*}
Integrating the above inequality over $B(0,r)$ and using \eqref{est:l2norms}, we finally obtain
\begin{equation}\label{eqn: tsloepersmooth}
\begin{split}
&\int_{B(0,r)} \rho_t |\partial_t \nabla P_t^*| \log^k_{+} (|\partial_t \nabla P_t^*|) \, dx\\
&\leq 2^{3(k-1)} \int_{B(0, r) }\rho_t |\nabla^2 P_t^*| \log^{2k}_+(|\nabla^2 P_t^*|) \, dx
+ C_2 \int_{B(0,R)}\rho_t | (\nabla^2 P_t^*)^{-1/2}\partial_t \nabla P_t^*|^2\, dx\\
&\leq 2^{3(k-1)} \int_{B(0,r)}\rho_t |\nabla^2 P_t^*|
\log^{2k}_+(|\nabla^2 P_t^*|) \, dx+ C_1 \cdot C_2,
\end{split}
\end{equation}
for all $0<r\leq R$.

\medskip
\emph{Step 2: The approximation argument.}
{We now consider the velocity field $U$ given by Theorem~\ref{thm:dualeq},}
we take a sequence of smooth  convex domains $\Omega_n$ which converges to $\Omega$ in the Hausdorff distance,
and a sequence $(\psi^n)\subset C_c^{\infty}(B(0,n))$ of cut off functions such that $0\leq\psi_n\leq 1$, $\psi^n(x) = 1$ inside $B(0,n/2)$, 
$|\nabla \psi^n| \leq 2/n$ in $\re^3$. Let us also consider a sequence of space-time mollifiers $\sigma^n$ with support
contained in $B(0,1/n)$ and a sequence 
of space mollifiers $\varphi^n$. We extend the function $U_t$ for $t\leq 0$ by setting $U_t=0$ for every $t<0$.

Let us consider a compactly supported space regularization of $\rho_0$ and a space-time regularization of $U$, namely
$$\rho^n_0:= \frac{(\rho_0\ast\varphi^n)}{c_n} 1_{B(0,n)}, \quad\quad U_t^n(x) := (U \ast \sigma^n) \psi^n, $$
where $c_n{ \uparrow 1}$ is chosen so that $\rho^n_0$ is a probability measure on $\R^3$.
Let $\rho^n_t$ be the solution of the continuity equation
$$\partial_t \rho^n_t +\nabla \cdot(U^n_t\rho^n_t) = 0 \quad\quad\text{in $\re^3\times [0,\infty)$}$$
with initial datum $\rho^n_0$. From the regularity of the velocity field $U^n_t$ and of the initial datum $\rho^n_0$ we have that 
$\rho^n \in C^\infty(B(0,{n}) \times [0,\infty))$.

Since $U_t$ is divergence-free and satisfies the inequality $|U_t(x)| \leq |x|+\diam$, we get
$$|U^n_t|(x) \leq |U \ast \sigma^n| (x) \leq \| U_t\|_{L^\infty(B(x, {1/n}))} \leq |x|+\diam +\frac 1n \leq |x|+\diam +1,$$
$$|\nabla \cdot U^n_t|(x) = |(U_t \ast \sigma^n)\cdot \nabla \psi^n|(x)\leq \frac{2 (n+{1+}\diam)}{n} \leq 3$$
for $n$ large enough.
Moreover, from the properties of $\rho_0$ we obtain that, for $n$ large enough,
$$\rho^n_0(x) \leq  \frac{2c_0}{(|x|-1/n)^K} \leq \frac{4 c_0}{|x|^K} \qquad \forall \,|x|\geq M+2,$$
$$\| \rho^n_0 \|_\infty \leq 2\| \rho_0\|_\infty \qquad \mbox{and}\qquad
\Big\|\frac{1}{ \rho^n_0}\Big \|_{L^\infty(B(0,n))} \leq \Big\| \frac{1}{ \rho_0}\Big \|_{L^\infty(B(0,{n+1}))}.$$

Hence the hypotheses of Lemma~\ref{lemma:dec-est} are satisfied with ${N}= 3$, $A=1$, $D= \diam +1$, ${d_0} = 4 c_0$.
Moreover $\rho^n_t$ vanishes outside $B(0,n)$,
and by \eqref{ts:supp-rho}
there exist {constants $\lambda_{n}:=e^{-3T}\Big\| \frac{1}{ \rho_0}\Big \|_{L^\infty(B(0,{n+1}))}^{-1}>0$, $\Lambda:=2e^{3T}\|\rho_0\|_\infty$}, and
$M_1$, $c_1$ depending on $T, M, c_0, \diam$ only, such that 
$$\lambda_{n,T}\leq\rho^n_t(x)\leq \Lambda \qquad\forall\, (x,t)\in B(0,n)\times [0,T],$$
$$\rho^n_t(x) \leq \frac{c_1}{|x|^K} \qquad \text{whenever $|x|\geq M_1$.}$$
(Observe that $\lambda_{n}$ depends on $n$, but the other constants are all independent of $n$.)
Thus, from Statement (ii) of Lemma \ref{lemma:dec-est} we get that, for all $r>0$,
\begin{equation}\label{eq:bb}
\rho^n_t(x) \ge e^{-3T}\inf\Big\{\rho^n_0(y):\ y\in B\Big(0,re^{t}+(\diam +1)[e^{t}-1]\Big)\Big\}
\qquad \forall\, (x,t) \in B(0,r)\times[0,T].
\end{equation}
If $n$ is large enough, the right-hand side of \eqref{eq:bb} is different from $0$,
and can be estimated from below in terms of $\rho_0$ by
$$\lambda=\lambda(r,T,\rho_0,\Omega):=e^{-3T}\inf\Big\{\rho_0(y):\ y\in B\Big(0,re^{t}+(\diam +1)[e^{t}-1]+1\Big)\Big\}>0.$$
Therefore, for any $r>0$ we can bound the density $\rho^n$ from below inside
$B(0,r)$ with a constant \textit{independent of $n$}:
\begin{equation}\label{eq:bound-below}
\lambda\leq\rho^n_t(x)\leq \Lambda \qquad\forall\, (x,t)\in B(0,r)\times [0,T].\end{equation}

Let now $P^{n*}_t$ be the unique convex function such that
$P^{n*}_t(0) = 0$ and 
$
(\nabla P^n_t)_\sharp \rho^n_t=\L_{\Omega_n}.
$
From the stability of solutions to the continuity equation  with $BV$ velocity field, \cite[Theorem 6.6]{AM1}, we infer that
\begin{equation}\label{eqn:transp-conv1}
\rho^n_t \to \rho_t\qquad\text{in
$L^1_{\rm loc} (\R^3)$, for any $t>0$,}
\end{equation}  
where $\rho_t$ is the unique solution of \eqref{eqn:dualsystem} corresponding to the velocity field $U$. 
Since $\Omega_n$ is converging to $\Omega$, from standard stability results for
optimal transport maps (see for instance \cite[Corollary 5.23]{V} and \cite[Section 4]{DeFi2}) it follows that
\begin{equation}\label{eqn:transp-conv}
\nabla P^{n*}_t \to \nabla P_t^* \quad\quad\mbox{in }L^1_{\rm loc} (\R^3)
\end{equation}
for any $t>0$.
Moreover, by Theorem \ref{thm:transport reg}(ii), Remark \ref{rmk:set-conv}, and \eqref{eq:bound-below}, for every $k\in\N$
\begin{equation}\label{eqn:lloglbound}
{ \limsup_{n\to \infty}}\int_{B(0,r)}\rho^n_t |\nabla^2 P^{n*}_t| \log^{2k}_+(|\nabla^2
P^{n*}_t|) \, dx  {\,<\infty\qquad\forall \,r>0},
\end{equation}
and by the stability theorem in the Sobolev topology
estabilished in \cite[Theorem 1.3]{DeFi2} it follows that
\begin{equation}\label{eqn:d2-conv-1}
 \lim_{n\to \infty} \int_{B(0,r)}\rho^n_t |\nabla^2 P^{n*}_t| \log^{2k}_+(|\nabla^2 P^{n*}_t|) \, dx
 =  \int_{B(0,r)}\rho_t |\nabla^2 P_t^*| \log^{2k}_+(|\nabla^2 P_t^*|) \, dx\qquad\forall \,r>0.
 \end{equation}
 
Since $(\rho_t^n,U_t^n)$ satisfy the assumptions \eqref{eqn:loep-reg1}
through \eqref{eqn:loep-reg6}, by Step 1 we can apply {\eqref{eqn: tsloepersmooth}} to $(\rho_t^n,U_t^n)$ to obtain
\begin{equation}
\label{ts:loep-n}
\int_{B(0, r)} \rho^n_t |\partial_t \nabla P_t^{n*}| \log^k_+ (|\partial_t \nabla P_t^{n*}|) \, dx
\leq 2^{3(k-1)} \int_{B(0, r)}\rho_t |\nabla^2 P_t^{n*}|
\log^{2k}_+(|\nabla^2 P_t^{n*}|) \, dx+ C
\end{equation}
for all $r<n$, where the constant $C$ does not depend on $n$. 
 
Let $\phi \in C^{\infty}_c((0,T))$ be a nonnegative function. From the Dunford-Pettis Theorem, taking into account \eqref{eqn:transp-conv1} and \eqref{eqn:transp-conv},
it is clear that $ \phi(t)\rho^n_t \partial_t \nabla P^{n*}_t$ converge weakly in $L^1(B(0,r)\times (0,T))$
to $\phi(t) \rho_t\partial_t \nabla P_t^*$. Moreover, since the function
$w\mapsto |w|\log_+^k(|w|/r)$ is convex for every $r\in (0,\infty)$,
 we can apply Ioffe lower semicontinuity theorem \cite[Theorem 5.8]{AFP}  to the functions $\phi(t)\rho^n_t \partial_t \nabla P^{n*}_t$ and $\phi(t)\rho^n_t$ to infer
\begin{equation}\label{eqn:time der}
\begin{split}
\int_0^T\! \phi(t)\int_{B(0,r)}\!\! \rho_t |\partial_t \nabla P_t^*| \log^k_{+} (|\partial_t \nabla P_t^*|) \, dx \, dt
\leq \liminf_{n\to\infty} \int_0^T\! \phi(t) \int_{B(0,r)}\!\! \rho^n_t
|\partial_t \nabla P^{n*}_t| \log^k_{+} (|\partial_t \nabla
P^{n*}_t|) \, dx \, dt.
\end{split}
\end{equation}

Taking \eqref{ts:loep-n}, \eqref{eqn:d2-conv-1},  \eqref{eqn:lloglbound}, and \eqref{eqn:time der} into account, by Lebesgue dominated convergence theorem we obtain 
\begin{multline*}
\int_0^T\phi(t)\int_{B(0,r)} \rho_t |\partial_t \nabla P_t^*| \log^k_+ (|\partial_t \nabla P_t^*|) \, dx \, dt\\
\leq \int_0^T \phi(t) \left(
2^{3(k-1)} \int_{B(0, r)}\rho_t |\nabla^2 P_t^{*}|
\log^{2k}_+(|\nabla^2 P_t^{*}|) \, dx+ C
\right) \, dt.
\end{multline*}
Since this holds for every $\phi \in C^{\infty}_c((0,T))$ nonnegative, by a localization argument we obtain the desired result.
\end{proof}

\begin{remark}\label{exp2}
Thanks to Remark \ref{exp1} one can prove that 
for every $T>0$ and $r>0$ there exist a constant $\kappa>1$ and a constant $C$ which depend on $r$, $\rho_0$, $T$, $\diam$ such that, for almost every $t \in [0,T]$ we have that $\nabla P_t^*\in W^{1,\kappa}(B(0,r)\times [0,\infty);\R^3)$ and
\begin{equation*}
\int_{B(0, r)} \rho_t |\partial_t \nabla P_t^*|^\kappa \, dx
\leq C.
\end{equation*}
This estimate provides better local integrability of the time derivative of $\nabla P_t^*$. The proof follows the same lines of Proposition \ref{prop:est} (see also \cite[Proposition 5.1]{Loe1}). However the exponent $\kappa$ is not universal, but depends in a nontrivial way from the local lower bounds on the density which are related to $r$, $\rho_0$ and $T$. Therefore we preferred to state  Proposition \ref{prop:est} with a universal modulus of integrability.

We finally point out that in the compact setting studied in \cite{acdf} the same argument provides a global $L^\kappa$ estimate of $\partial_t \nabla P_t^*$ on the torus, with $\kappa$ depending only on the upper and lower bound on $\rho_0$, which is also uniform in time.
\end{remark}

\section{Existence of an Eulerian solution}\label{sect:proof thm}

\begin{proof}[Proof of Theorem~\ref{thm:main}]
First of all notice that by statement (ii) of Theorem~\ref{thm:transport reg} and Proposition~\ref{prop:est}, it holds
$|\nabla^2
P_t^*|,\,|\partial_t \nabla P^*_t|\in L_{\rm
loc}^\infty([0,\infty),L^1_{\rm loc}(\R^3))$. Moreover, since $(\nabla
P_t)_\sharp \L^3=\rho_t$, it is immediate to check the
function $u$ in \eqref{defn:ut} is well-defined and $|u|$ belongs
to $L_{\rm loc}^\infty([0,\infty),L^1_{\rm loc}(\R^3))$.

Let $\phi\in C^\infty_c(\Omega\times[0,\infty))$ be a test
function and let us consider $\varphi:
\re^3\times [0,\infty)\to\re^3$ given by
\begin{equation}\label{defn:test}
\varphi_t(y) := y \phi_t(\nabla P_t^*(y)).
\end{equation}
Clearly $\varphi$ is compactly supported in time because so is $\phi$; moreover $P_t$ are Lipschitz on $\supp\,\phi_t$ as $t$ varies in any compact subset of $[0,\infty)$ with bounded Lipschitz constants. Hence the set $\nabla P_t(\supp\,\phi_t)$, which contains $\supp\,\varphi_t$, is bounded in space. Therefore $\phi_t$ is compactly supported in $\R^3 \times [0,\infty)$.
Moreover, Proposition~\ref{prop:est} implies that $\varphi\in
W^{1,1}(\re^3\times[0,\infty))$. So, by Lemma~\ref{rmk:scontr},
each component of the function $\varphi_t(y)$ is an admissible test
function for \eqref{eqn:sg-dual-weak}. For later use, we write down
explicitly the derivatives of $\varphi$:
\begin{equation}\label{eqn:ber1}
\begin{cases}
\partial_t\varphi_t(y)=y[\partial_t\phi_t](\nabla P_t^*(y)) + y \bigl([\nabla\phi_t](P_t^*(y))\cdot
\partial_t\nabla P_t^*(y)\bigr),
\\
\nabla\varphi_t(y)=Id \,\phi_t(\nabla P_t^*(y))+
y\otimes \bigl([\nabla^T\phi_t](P_t^*(y))\nabla^2P_t^*(y)\bigr).
\end{cases}
\end{equation}
Taking into account that $(\nabla P_t)_\sharp \L_{\Omega}=\rho_t\L^3$
and that $[\nabla P_t^*](\nabla P_t(x))=x$ almost everywhere, we can rewrite the
boundary term in \eqref{eqn:sg-dual-weak} as
\begin{equation}
\label{eqn:test-1} \int_{\R^3} \varphi_0(y) \rho_0(y) \, dy
 =  \int_{\Omega} \nabla {P_0}(x)\phi_0(x) \, dx.
\end{equation}
In the same way, since $U_t(y)=J(y-\nabla P_t^*(y))$, we can use
\eqref{eqn:ber1} to rewrite the other term as
\begin{equation}
\label{eqn:test-12}
\begin{split}
\int_0^\infty \int_{\R^3} \Big\{ \partial_t \varphi_t(y) +&
\nabla \varphi_t(y) \cdot U_t(y)\Big\} \rho_t(y)\, dy \, dt\\
& = \int_0^\infty  \int_{\Omega}\Big\{ \nabla P_t(x) \partial_t\phi_t(x)
+ \nabla P_t(x) \bigl(\nabla\phi_t(x)\cdot [\partial_t \nabla P_t^*]( \nabla P_t(x))\bigr)\\
&+ \bigl[ Id \,\phi_t(x) + \nabla
P_t(x) \otimes \bigl(\nabla^T\phi_t(x)\nabla^2 P_t^*( \nabla
P_t(x))\bigr)\bigr]J(\nabla P_t(x)- x) \Big\}\, dx \, dt
\end{split}
\end{equation}
which, taking into account the formula \eqref{defn:ut} for $u$,
after rearranging the terms turns out to be equal to
\begin{equation}\label{eqn:test-2}
 \int_0^\infty\int_{\Omega}
 \nabla P_t(x) \Big\{\partial_t \phi_t(x) + u_t(x)\cdot \nabla
 \phi_t(x)\Big\}+J \Big\{\nabla P_t(x)-x \Big\} \phi_t(x)\, dx \, dt.
\end{equation} 
Hence, combining (\ref{eqn:test-1}), \eqref{eqn:test-12},
(\ref{eqn:test-2}), and (\ref{eqn:sg-dual-weak}), we obtain the validity of
(\ref{eqn:sg1-weak}).

Now we prove (\ref{eqn:sg2-weak}). Given $\phi\in
C_c^\infty(0,\infty)$ and $\psi\in
C^\infty_c({\Omega})$, let us consider $\varphi:\re^3\times
[0,\infty)\to \re$ defined by
\begin{equation}
\varphi_t(y) := \phi(t)\psi(\nabla P_t^*(y)).
\label{defn:test-div}
\end{equation}
As in the previous case, $\varphi\in W^{1,1}(\R^3\times [0,\infty))$ and is compactly supported in time and space, so
we can use $\varphi$ as a test function in
(\ref{eqn:sg-dual-weak}). Then, identities analogous to
\eqref{eqn:ber1} yield
\begin{equation*}
\begin{split}
0& = \int_0^{\infty}\int_{\R^3} \left\{ \partial_t \varphi_t(y) +
\nabla \varphi_t(y) \cdot U_t(y)\right\} \rho_t(y)\, dy \, dt\\
& =\int_0^{\infty} \phi'(t)\int_{\Omega} \psi(x)\, dx \, dt\\
&\phantom{A}
 +  \int_0^{\infty} \phi(t)
 \int_{\Omega}\Big\{\nabla\psi(x)\cdot \partial_t \nabla P_t^*( \nabla P_t(x))
  +\nabla^T \psi(x)\nabla^2 P_t^*( \nabla P_t(x))J( \nabla P_t(x) - x) \Big\}\, dx \, dt\\
& = \int_0^{\infty} \phi(t) \int_{\Omega} \nabla \psi(x)\cdot u_t(x)
\, dx \, dt.
\end{split}
\end{equation*}
Since $\phi$ is arbitrary we obtain
$$
\int_{\Omega } \nabla \psi(x)\cdot u_t(x) \, dx =0 \qquad \mbox{for
a.e. $t>0$.}
$$
By a standard density argument it follows that the above equation holds
outside a negligible set of times
independent of the test function $\psi$, thus proving \eqref{eqn:sg2-weak}.
\end{proof}

\end{document}